\documentclass[11pt,oneside,english]{amsart}
\usepackage[T1]{fontenc}
\usepackage[latin9]{inputenc}
\usepackage{babel}
\usepackage{amstext}
\usepackage{amsthm}
\usepackage{amssymb}
\usepackage[unicode=true,pdfusetitle,
 bookmarks=true,bookmarksnumbered=false,bookmarksopen=false,
 breaklinks=false,pdfborder={0 0 1},backref=false,colorlinks=false]
 {hyperref}

\makeatletter
\numberwithin{equation}{section}
\numberwithin{figure}{section}
\theoremstyle{plain}
\newtheorem{thm}{\protect\theoremname}
\theoremstyle{plain}
\newtheorem{cor}{\protect\corollaryname}
\theoremstyle{plain}
\newtheorem{lem}{\protect\lemmaname}

\makeatother

\providecommand{\corollaryname}{Corollary}
\providecommand{\lemmaname}{Lemma}
\providecommand{\theoremname}{Theorem}

\begin{document}
\title{Endpoint mixed weak type extrapolation}
\author{Sheldy Ombrosi}
\address{(Sheldy Ombrosi) Departamento de Análisis Matemático y Matemática
Aplicada. Facultad de Ciencias Matemáticas. Universidad Complutense
(Madrid, Spain). Departamento de Matemática e Instituto de Matemática
de Bahía Blanca. Universidad Nacional del Sur - CONICET Bahía Blanca,
Argentina.}
\email{sombrosi@ucm.es}
\author{Israel P. Rivera-Ríos}
\address{(Israel P. Rivera Ríos) Departamento de Análisis Matemático, Estadística
e Investigación Operativa y Matemática Aplicada. Facultad de Ciencias.
Universidad de Málaga (Málaga, Spain). Departamento de Matemática.
Universidad Nacional del Sur (Bahía Blanca, Argentina).}
\email{israelpriverarios@uma.es}
\thanks{Both authors were partially supported by FONCyT PICT 2018-02501. The
first author was partially supported by Spanish Government Ministry
of Science and Innovation through grant PID2020-113048GB-I00. The
second author was partially supported as well by FONCyT PICT 2019-00018
and by Junta de Andalucía UMA18FEDERJA002 and FQM-354.}
\begin{abstract}
The purpose of this note is to extend the extrapolation result in
\cite{CUMP} as follows. Given a family $\mathcal{F}$ of pairs of
functions suppose that for some $0<p<\infty$ and for every $w\in A_{\infty}$
\begin{equation}
\int f^{p}w\leq c_{w}\int g^{p}w\qquad(f,g)\in\mathcal{F}\label{eq:Hip-1}
\end{equation}
provided the left-hand side of the estimate is finite. If we have
that $B(t)=\frac{t}{\log(e+1/t)^{\rho}}$ for some $\rho>0$, then,
for every $u\in A_{1}$ and every $v\in A_{\infty}$ we have that
\[
\left\Vert \frac{f}{v}\right\Vert _{L^{B,\infty}(uv)}\lesssim\left\Vert \frac{g}{v}\right\Vert _{L^{B,\infty}(uv)},
\]
where 
\[
L^{A,\infty}(uv)=\inf\left\{ \lambda>0:\sup_{t>0}A(t)uv\left(\left\{ x\in\mathbb{R}:|f(x)|>\lambda t\right\} \right)\leq1\right\} 
\]
is the weak Orlicz type introduced in \cite{Ia}. As a corollary of
this extrapolation result we derive a mixed weak type inequality for
Coifman-Rochberg-Weiss commutators.
\end{abstract}

\maketitle

\section{Introduction and main results}

We recall that $w\in A_{1}$ if 
\[
[w]_{A_{1}}=\left\Vert \frac{Mw}{w}\right\Vert _{L^{\infty}}<\infty.
\]
 In the late seventies, Muckenhoupt and Wheeden \cite{MW} showed
that given a weight $w\in A_{1}$ for
\[
\left|\left\{ x\in\mathbb{R}\,:\,w(x)|Gf(x)|>t\right\} \right|\lesssim\frac{1}{t}\int|f(x)|w(x)dx
\]
where $G$ is either the Hilbert transform or the Hardy-Littlewood
maximal function. Almost a decade later, Sawyer \cite{S} generalized
that result in the following way. Given $u,v\in A_{1}$ 
\begin{equation}
uv\left(\left\{ x\in\mathbb{R}\,:\,\frac{M(fv)(x)}{v(x)}>t\right\} \right)\lesssim\frac{1}{t}\int|f(x)|u(x)v(x)dx.\label{eq:Sawyer}
\end{equation}
Note that even though $u,v\in A_{1}$, $uv$ could not have good regularity
properties. Also the perturbation induced by having the weight $v$
dividing the maximal function in the level set makes this inequality
way harder to settle than the in the case $v=1$. In \cite{S} it
was conjectured as well whether (\ref{eq:Sawyer}) holds with $M$
replaced by the Hilbert transform. This conjecture was solved in the
positive by Cruz-Uribe, Martell and Pérez \cite{CUMP}. In that paper
generalized (\ref{eq:Sawyer}) to the case $n\geq1$ and showed that
(\ref{eq:Sawyer}) actually holds for general Calderón-Zygmund operators.
They settled the latter estimate via an extrapolation type result.
The precise statement of their result is the following.
\begin{thm}
\label{Thm:EndpointCoifmanCUMP}Given a family $\mathcal{F}$ suppose
that for some $0<p<\infty$ and for every $w\in A_{\infty}$ 
\begin{equation}
\int f^{p}w\leq c_{w}\int g^{p}w\qquad(f,g)\in\mathcal{F}\label{eq:CFCUMP}
\end{equation}
provided the left-hand side of the estimate is finite. Then, if $u\in A_{1}$
and $v\in A_{\infty},$
\begin{equation}
\left\Vert \frac{f}{v}\right\Vert _{L^{1,\infty}(uv)}\lesssim\left\Vert \frac{g}{v}\right\Vert _{L^{1,\infty}(uv)}.\label{eq:extrapCUMP}
\end{equation}
\end{thm}
We remit the reader to to Subsection \ref{subsec:ApOav} for the precise
definition of $A_{\infty}$. Note that the classical Coifman-Fefferman
inequality says that (\ref{eq:CFCUMP}) holds for any $0<p<\infty$
with $(f,g)=(Tf,Mf)$ where $T$ is any Calderón-Zygmund operator
and $M$ is the Hardy-Littlewood maximal operator. Hence the result
above yields
\[
\left\Vert \frac{Tf}{v}\right\Vert _{L^{1,\infty}(uv)}\lesssim\left\Vert \frac{Mf}{v}\right\Vert _{L^{1,\infty}(uv)}
\]
which combined with (\ref{eq:Sawyer}), leads to the corresponding
result for $T$. At this point it is worth noting that nowadays there
is no known proof that avoids the use of extrapolation in the case
of Calderón-Zygmund operators for $u\in A_{1}$ and $v\in A_{1}$,
however a direct proof is feasible in the case $u\in A_{1}$ and $v\in A_{\infty}(u)$
(see Subsection \ref{subsec:ApOav} for the precise definition of
$A_{\infty}(u)$). Also the statement of the result above led the
authors to raise the so called Sawyer's conjecture, namely whether
(\ref{eq:Sawyer}) could hold assuming that $u\in A_{1}$ and $v\in A_{\infty}$.
That conjecture was positively solved recently by Li, the second author
and Pérez \cite{LiOP}.

Now we turn our attention to our contribution. Note that a few years
ago, the study of mixed weak type inequalities for commutators was
begun by Berra, Carena and Pradolini \cite{BCP} (see as well \cite{CRR}).
There the following result was settled. 
\begin{thm}
\label{Thm:CommBCP}Let $u\in A_{1}$ and $v\in A_{\infty}(u)$. If
$T$ is a Calderón-Zygmund operator and $b\in BMO$, then 
\begin{equation}
uv\left(\left\{ x\in\mathbb{R}\,:\,\frac{|[b,T](fv)(x)|}{v(x)}>t\right\} \right)\lesssim\int\Phi_{1}\left(\frac{\|b\|_{BMO}|f(x)|}{t}\right)u(x)v(x)dx\label{eq:CommBCP}
\end{equation}
where $\Phi_{\rho}(t)=t\log^{\rho}\left(e+t\right)$.
\end{thm}
No positive results are known for ``unrelated'' weights, namely
assuming just that $u\in A_{1}$ and $v\in A_{1}$. A reasonable approach
to this question may consist in establishing the corresponding estimate
for $M_{L\log L}$ (see Subsection \ref{subsec:ApOav} for the precise
definition), namely
\[
uv\left(\left\{ x\in\mathbb{R}\,:\,\frac{|M_{L\log L}(fv)(x)|}{v(x)}>t\right\} \right)\lesssim\int\Phi\left(\frac{\|b\|_{BMO}|f(x)|}{t}\right)u(x)v(x)dx
\]
and then reducing the problem for commutators to the estimate for
$M_{L\log L}$ via some counterpart of \ref{Thm:EndpointCoifmanCUMP}.
In fact, we are going to provide such a counterpart. In order to provide
our statement we need following definition that we borrow from \cite{Ia}. 

Given an strictly increasing function $A:[0,\infty)\rightarrow[0,\infty)$
and a weight $w$ we define,
\[
\left\Vert f\right\Vert _{L^{A,\infty}(w)}=\inf\left\{ \lambda>0:\sup_{t>0}A(t)w\left(\left\{ x\in\mathbb{R}:|f(x)|>\lambda t\right\} \right)\leq1\right\} .
\]
As we will see in full detail in \ref{subsec:WOrliczSpaces}, for
suitable $A$, 
\begin{equation}
\mu\left(\left\{ |G(x)|>t\right\} \right)\leq c_{G}\int A\left(\frac{|f|}{t}\right)d\mu\iff\|Gf\|_{L^{B,\infty}(d\mu)}\leq\tilde{c}_{G}\|f\|_{A(d\mu)}\label{eq:ReWriteIntro}
\end{equation}
where $B(t)=\frac{1}{A\left(\frac{1}{t}\right)}$ and 
\[
\|f\|_{A(d\mu)}=\inf\left\{ \lambda>0:\int_{X}A\left(\frac{|f|}{\lambda}\right)d\mu\leq1\right\} .
\]
Armed with this notation we are in the position to state our main
result.
\begin{thm}
\label{Thm:EndpointCoifman}Given a family $\mathcal{F}$ suppose
that for some $0<p<\infty$ and for every $w\in A_{\infty}$ 
\begin{equation}
\int f^{p}w\leq c_{w}\int g^{p}w\qquad(f,g)\in\mathcal{F}\label{eq:Hip}
\end{equation}
provided the left-hand side of the estimate is finite. Let $\Phi_{\rho}(t)=t\log(e+t)^{\rho}$
for some $\rho>0$. Then, for every $u\in A_{1}$ and every $v\in A_{\infty}$
we have that 
\[
\left\Vert \frac{f}{v}\right\Vert _{L^{B_{\rho},\infty}(uv)}\lesssim\left\Vert \frac{g}{v}\right\Vert _{L^{B_{\rho},\infty}(uv)}
\]
where $B_{\rho}(t)=\frac{1}{\Phi_{\rho}\left(\frac{1}{t}\right)}=\frac{t}{\log(e+\frac{1}{t})^{\rho}}.$
\end{thm}
The fact that we can rewrite the inequalities in terms of weak Orlicz
norms and results generalizing associated spaces and interpolation
to that scale, relying upon some ideas \cite{Ia} and \cite{BS},
will be fundamental for our purposes. This change of view and its
corresponding ``toolbox'' is the crux of the paper. Having all the
aforementioned elements at our disposal will enable us to follow the
strategy in \cite{CUMP} to finally settle Theorem \ref{Thm:EndpointCoifman}.
Before continuing our discusion, we would like to note that, actually,
a similar argument to the one that will provide in the proof of Theorem
\ref{Thm:EndpointCoifman} allows as well to obtain a more general
version of the theorem above with the hypothesis $v\in A_{\infty}$
replaced by $v^{\theta}\in A_{\infty}$ for some $\theta>0$ hence
extending results in \cite{OP}. Note that the flexibility provided
by that generalized hypothesis, is useful, for instance, in the multilinear
setting. We remit the reader to \cite{LOPi} for further details.

We devote the following lines to provide an example of application
of our main result. Note that if $b\in BMO$ and $T$ is a Calderón-Zygmund
operator, it is well known (see for instance \cite{P}) that, if $\|b\|_{BMO}=1$,
\[
\int|[b,T]f(x)|^{p}w(x)dx\leq c_{w}\int M_{L\log L}f(x)^{p}w(x)dx\qquad p\in(0,\infty),\,w\in A_{\infty}.
\]
Then Theorem \ref{Thm:EndpointCoifmanCUMP} above yields that if $\Phi$$(t)=t\log(e+t)^{\rho}$
then, 
\[
\left\Vert \frac{[b,T](fv)}{v}\right\Vert _{L^{B_{\rho},\infty}(uv)}\lesssim\left\Vert \frac{M_{L\log L}(fv)}{v}\right\Vert _{L^{B_{\rho},\infty}(uv)}.
\]
Bearing that estimate in mind, note that if we had that
\begin{equation}
uv\left(\left\{ \frac{M_{L\log L}(fv)}{v}>t\right\} \right)\lesssim\int\Phi_{\rho}\left(\frac{|f|}{t}\right)uv\label{eq:MLlogLPHI}
\end{equation}
then we that would imply, via (\ref{eq:ReWriteIntro}), that,
\[
\left\Vert \frac{M_{L\log L}(fv)}{v}\right\Vert _{L^{B_{\rho},\infty}(uv)}\lesssim\|f\|_{\Phi_{\rho}(uv)}
\]
and consequently
\[
\left\Vert \frac{[b,T](fv)}{v}\right\Vert _{L^{B_{\rho},\infty}(uv)}\lesssim\left\Vert \frac{M_{L\log L}(fv)}{v}\right\Vert _{L^{B_{\rho},\infty}(uv)}\lesssim\|f\|_{\Phi_{\rho}(uv)}
\]
which in turn, applying the equivalence in (\ref{eq:ReWriteIntro})
again, would allow us to conclude that
\[
uv\left(\left\{ \frac{|[b,T](fv)|}{v}>t\right\} \right)\lesssim\int\Phi_{\rho}\left(\frac{|f|}{t}\right)uv.
\]

Up until now the best known counterpart of (\ref{eq:MLlogLPHI}) is
contained in \cite{B} and it says that (\ref{eq:ReWriteIntro}) holds
for $\Phi_{1+\frac{1}{\varepsilon}}(t)$ with $0<\varepsilon<\varepsilon_{0}$
for some $\varepsilon_{0}$ depending on the weights involved. Hence
the best possible available result for commutators up until now is
the following.
\begin{cor}
Let $T$ be a Calderón-Zygmund operator and $b\in BMO$. If $u\in A_{1}$
and $v\in A_{\infty}$ we have that
\[
uv\left(\left\{ \frac{|[b,T](fv)|}{v}>t\right\} \right)\lesssim\int_{\mathbb{R}^{n}}\Phi_{1+\frac{1}{\varepsilon}}\left(\frac{|f|\|b\|_{BMO}}{t}\right)uv
\]
for some $0<\varepsilon<\varepsilon_{0}$ depending on the weights.
\end{cor}
The remainder of the paper is organized as follows. In Section \ref{sec:Lemmatta}
we gather all the Lemmatta required to settle the main result, which
proof is presented in Section \ref{sec:ProofMain}.

\section{\label{sec:Lemmatta}Lemmatta}

\subsection{$A_{p}$ weights and Orlicz averages\label{subsec:ApOav}}

We recall that if $p>1$, we have that $w\in A_{p}$ if and only if
\[
[w]_{A_{p}}=\sup_{Q}\frac{1}{|Q|}\int_{Q}w\left(\frac{1}{|Q|}\int_{Q}w^{-\frac{1}{p-1}}\right)^{p-1}<\infty.
\]
Since the $A_{p}$ classes are increasing with $p$ it is natural
to define $A_{\infty}=\bigcup_{p\geq1}A_{p}$. 

We recall that $w\in A_{p}(u)$ if 
\begin{align*}
[w]_{A_{p}(u)} & =\sup_{Q}\frac{1}{u(Q)}\int_{Q}wu\left(\frac{1}{u(Q)}\int_{Q}w^{-\frac{1}{p-1}}u\right)^{p-1}<\infty.\qquad\text{if }1<p<\infty.\\{}
[w]_{A_{1}(u)} & =\left\Vert \frac{\sup\frac{1}{u(Q)}\int_{Q}wu}{w}\right\Vert _{L^{\infty}}<\infty.\qquad\text{if }p=1.
\end{align*}
Analogously, $A_{\infty}(u)=\bigcup_{p\geq1}A_{p}(u).$

We borrow the following Lemma from \cite[Lemma 2.3]{CUMP} since it
will be useful for our interests.
\begin{lem}
\label{Lem:CUMP}If $u\in A_{1}$, $v\in A_{p}$, $1\leq p<\infty$,
then there exists $0<\varepsilon_{0}<1$ depending only on $[u]_{A_{1}}$
such that $uv^{\varepsilon}\in A_{p}$ for all $0<\varepsilon<\varepsilon_{0}$.
\end{lem}
We end this subsection recalling that given a Young function $A$
we can define the average of $f$ associated to $A$ by
\[
\|f\|_{A(L),Q}=\|f\|_{A,Q}=\inf\left\{ \lambda>0:\frac{1}{|Q|}\int_{Q}A\left(\frac{x}{\lambda}\right)dx\leq1\right\} .
\]
Given a Young function $A$, it is natural to define the maximal operator
associated to $A$ by 
\[
M_{A(L)}f(x)=M_{A}f(x)=\sup_{x\in Q}\|f\|_{A,Q.}
\]
 A particular case of interest for us will be the case $M_{L\log L}$
which is given by $A(t)=t\log(e+t)$.

\subsection{\label{subsec:WOrliczSpaces}Results on weak Orlicz spaces}

Our first result contains the precise statement of \ref{eq:ReWriteIntro}.
\begin{lem}
\label{Lem:rewrite}Let $A:[0,\infty)\rightarrow[0,\infty)$ be an
strictly increasing submultiplicative function. Then given a linear
or a sublinear operator $G$, the following statements are equivalent
\begin{enumerate}
\item There exists $c_{G}>0$ such that for every $t>0$
\[
\mu\left(\left\{ |G(x)|>t\right\} \right)\leq c_{G}\int A\left(\frac{|f|}{t}\right)d\mu.
\]
\item There exists $\tilde{c}_{G}>0$ such that 
\[
\|Gf\|_{L^{B,\infty}(d\mu)}\leq\tilde{c}_{G}\|f\|_{A(d\mu)}
\]
where $B(t)=\frac{1}{A\left(\frac{1}{t}\right)}$.
\end{enumerate}
\end{lem}
\begin{proof}
Assume that $(1)$ holds. Since $A$ is submultiplicative we have
that
\[
\mu\left(\left\{ |G(x)|>t\lambda\right\} \right)\leq c_{G}\int A\left(\frac{|f|}{t\lambda}\right)d\mu\leq c_{G}c_{A}A\left(\frac{1}{t}\right)\int A\left(\frac{|f|}{\lambda}\right)d\mu.
\]
Hence, 
\[
\frac{1}{A\left(\frac{1}{t}\right)}\mu\left(\left\{ |G(x)|>t\lambda\right\} \right)\leq c_{G}c_{A}\int A\left(\frac{|f|}{\lambda}\right)d\mu.
\]
Now choosing $\lambda=\lambda_{0}\|f\|_{A(d\mu)}$ for some $\lambda_{0}>0$
to be chosen we have that
\begin{align*}
\frac{1}{A\left(\frac{1}{t}\right)}\mu\left(\left\{ |G(x)|>t\lambda_{0}\|f\|_{A(d\mu)}\right\} \right) & \leq c_{G}c_{A}\int A\left(\frac{|f|}{\lambda_{0}\|f\|_{A(d\mu)}}\right)d\mu\\
 & \leq c_{G}c_{A}^{2}A\left(\frac{1}{\lambda_{0}}\right)\int A\left(\frac{|f|}{\|f\|_{A(d\mu)}}\right)d\mu\\
 & \leq c_{G}c_{A}^{2}A\left(\frac{1}{\lambda_{0}}\right)
\end{align*}
Taking $\lambda_{0}=\frac{1}{A^{-1}\left(\frac{1}{c_{G}c_{A}^{2}}\right)}$
we have that 
\[
\frac{1}{A\left(\frac{1}{t}\right)}\mu\left(\left\{ |G(x)|>t\lambda_{0}\|f\|_{A(d\mu)}\right\} \right)\leq1
\]
And this yields that 
\[
\|Gf\|_{L^{B,\infty}(d\mu)}\leq\lambda_{0}\|f\|_{A(d\mu)}.
\]

Now we assume that
\[
\|Gf\|_{L^{B,\infty}(d\mu)}\leq\tilde{c}_{G}\|f\|_{A(d\mu)}.
\]
Note that the estimate above in particular implies that
\[
\sup_{t>0}\frac{1}{A\left(\frac{1}{t}\right)}\mu\left(\left\{ x\in\mathbb{R}:|Gf(x)|>\tilde{c}_{G}\|f\|_{A(d\mu)}t\right\} \right)\leq1
\]
and hence
\[
\frac{1}{A(1)}\mu\left(\left\{ x\in\mathbb{R}:|Gf(x)|>\tilde{c}_{G}\|f\|_{A(d\mu)}\right\} \right)\leq1.
\]
Now we observe that by definition
\begin{align*}
1 & \leq\int A\left(\frac{2|f|}{\|f\|_{A(d\mu)}}\right)d\mu=\int A\left(\frac{\tilde{c}_{G}|f|}{\tilde{c}_{G}\|f\|_{A(d\mu)}}\right)d\mu\\
 & \leq c_{A}A(\tilde{c}_{G})\int A\left(\frac{|f|}{\tilde{c}_{G}\|f\|_{A(d\mu)}}\right)d\mu.
\end{align*}
Consequently,
\[
w\left(\left\{ x\in\mathbb{R}:|Gf(x)|>\tilde{c}_{G}\|f\|_{A(d\mu)}\right\} \right)\leq c_{A}A(\tilde{c}_{G})A(1)\int A\left(\frac{|f|}{\tilde{c}_{G}\|f\|_{A(d\mu)}}\right)d\mu.
\]
and by homogeneity we are done.
\end{proof}
A fundamental fact for our purposes is that we can rescale weak Orlicz
seminorms.
\begin{lem}
\label{Lem:rescaling}If $r>1$ then
\[
\|f\|_{L^{A,\infty}}=\||f|^{\frac{1}{r}}\|_{L^{A_{r},\infty}}^{r}
\]
where $A_{r}(t)=A(t^{r})$.
\end{lem}
\begin{proof}
We observe that
\begin{align*}
\left\Vert f\right\Vert _{L^{A,\infty}(w)} & =\inf\left\{ \lambda>0:\sup_{t>0}A(t)w\left(\left\{ x\in\mathbb{R}:|f(x)|>\lambda t\right\} \right)\leq1\right\} \\
 & =\inf\left\{ \lambda>0:\sup_{t>0}A(t)w\left(\left\{ x\in\mathbb{R}:|f(x)|^{\frac{1}{r}}>\lambda^{\frac{1}{r}}t^{\frac{1}{r}}\right\} \right)\leq1\right\} \\
 & =\inf\left\{ \lambda^{\frac{1}{r}}>0:\sup_{t>0}A(s^{r})w\left(\left\{ x\in\mathbb{R}:|f(x)|^{\frac{1}{r}}>\lambda^{\frac{1}{r}}s\right\} \right)\leq1\right\} ^{r}\\
 & =\inf\left\{ \lambda>0:\sup_{t>0}A(s^{r})w\left(\left\{ x\in\mathbb{R}:|f(x)|^{\frac{1}{r}}>\lambda s\right\} \right)\leq1\right\} ^{r}\\
 & =\||f|^{\frac{1}{r}}\|_{L^{A_{r},\infty}}^{r}
\end{align*}
\end{proof}

\subsubsection{Associated weak Orlicz spaces}

We recall that a Young function $A$ is of lower type $p>0$ if for
every $0\leq s\leq1$ we have that
\[
A(st)\leq Cs^{p}A(t).
\]
Note that in \cite[(4.8) Theorem]{Ia} it was shown that if $A$ is
of lower type greater than $1$ then
\[
\left\Vert g\right\Vert _{L^{A,\infty}(\mu)}\simeq\sup_{s>0}\frac{g_{\mu}^{*}(s)}{A^{-1}\left(\frac{1}{s}\right)}
\]
which is a norm. Following as well \cite{Ia} we define
\[
\|g\|_{L^{B,1}(\mu)}=\int_{0}^{\infty}\frac{g_{\mu}^{*}(s)}{B^{-1}(\frac{1}{s})}\frac{ds}{s}.
\]
Our next result provides conditions relating an Orlicz weak space
to its associated space.
\begin{lem}
\label{lem:assoc}Assume that $A$ is of lower type greater than $1$
and that for a certain Young function $B$ $A^{-1}(t)B^{-1}(t)\simeq t$.
Then
\[
\|g\|_{L^{(A,\infty)'}(\mu)}\simeq\|g\|_{L^{B,1}(\mu)}.
\]
\end{lem}
\begin{proof}
Note that since $A^{-1}(t)B^{-1}(t)\simeq t$ then $A^{-1}(\frac{1}{t})B^{-1}(\frac{1}{t})\simeq\frac{1}{t}$.\\
We define
\[
\|g\|_{L^{(A,\infty)'}(uv)}=\sup_{\left\Vert f\right\Vert _{L^{A,\infty}(uv)}=1}\left|\int_{\mathbb{R}^{d}}f(x)g(x)u(x)v(x)dx\right|
\]
Let
\[
\|g\|_{L^{B,1}(uv)}=\int_{0}^{\infty}\frac{g_{uv}^{*}(s)}{B^{-1}(\frac{1}{s})}\frac{ds}{s}
\]
First we note that
\begin{align*}
\left|\int fguvdx\right| & \leq\int_{0}^{\infty}f_{uv}^{*}(s)g_{uv}^{*}(s)ds=\int_{0}^{\infty}\frac{f_{uv}^{*}(s)g_{uv}^{*}(s)}{\frac{1}{s}}\frac{ds}{s}\\
 & \simeq\int_{0}^{\infty}\frac{f_{uv}^{*}(s)g_{uv}^{*}(s)}{A^{-1}(\frac{1}{s})B^{-1}(\frac{1}{s})}\frac{ds}{s}\\
 & \leq\left(\sup_{s>0}\frac{f_{uv}^{*}(s)}{A^{-1}(\frac{1}{s})}\right)\int_{0}^{\infty}\frac{g_{uv}^{*}(s)}{B^{-1}(\frac{1}{s})}\frac{ds}{s}\\
 & =\left\Vert f\right\Vert _{L^{A,\infty}(uv)}\|g\|_{L^{B,1}(uv)}
\end{align*}
And taking norm over $\|f\|_{L^{A,\infty}(uv)}=1$
\[
\|g\|_{L^{(A,\infty)'}(uv)}\leq\|g\|_{L^{B,1}(uv)}
\]
Now we have to show that 
\[
\|g\|_{L^{B,1}(uv)}\lesssim\|g\|_{L^{(A,\infty)'}(uv)}
\]
By the Luxemburg representation theorem \cite[Theorem II.4.10]{BS}
it suffices to deal with the measure space $(\mathbb{R}^{+},m)$ and
functions $g$ on $\mathbb{R}^{+}$ such that $g=g^{*}$. Let
\[
f(s)=A^{-1}\left(\frac{1}{s}\right).
\]
Note that since $f$ is decreasing, then $f=f^{*}$. Now we observe
that since
\[
A^{-1}\left(\frac{1}{t}\right)\simeq\frac{1}{tB^{-1}\left(\frac{1}{t}\right)}
\]
we have that
\begin{align*}
\|g\|_{L^{B,1}(uv)} & =\int_{0}^{\infty}\frac{g^{*}(s)}{B^{-1}(\frac{1}{s})}\frac{ds}{s}\simeq\int_{0}^{\infty}A^{-1}\left(\frac{1}{s}\right)g^{*}(s)ds\\
 & =\int_{0}^{\infty}f(s)g^{*}(s)ds=\int_{0}^{\infty}f^{*}(s)g^{*}(s)ds\\
 & \leq\|f\|_{L^{A,\infty}}\|g\|_{L^{(A,\infty)'}}.
\end{align*}
But since
\[
\|f\|_{L^{A,\infty}(uv)}=\sup_{s>0}\frac{f^{*}(s)}{A^{-1}(\frac{1}{s})}=\sup_{s>0}\frac{A^{-1}\left(\frac{1}{s}\right)}{A^{-1}(\frac{1}{s})}=1
\]
we are done.
\end{proof}

\subsubsection{An interpolation result in Orlicz scale}

In our next Theorem we show that it is possible to obtain results
via interpolation in the weak Orlicz scale having Lorentz spaces and
the $L^{\infty}\rightarrow L^{\infty}$ bound as departing points.
Since we are not aware this result is already available in the literature,
we provide a full proof. For related results we remit the interested
reader to \cite{C,CUH}.
\begin{thm}
\label{Thm:Interpolation}Assume that 
\[
\|Tf\|_{L^{p_{0},\infty}}\leq C_{0}\|f\|_{L^{p_{0},1}}
\]
 and 
\[
\|Tf\|_{L^{\infty}}\leq C_{1}\|f\|_{L^{\infty}}
\]
and that $A$ is a Young function such that 
\begin{equation}
\int_{s}^{\infty}\frac{1}{A^{-1}(\frac{1}{t})}t^{-\frac{1}{p_{0}}}\frac{dt}{t}\leq\kappa\frac{s^{-\frac{1}{p_{0}}}}{A^{-1}(\frac{1}{s})}\qquad s>0\label{eq:pwisecond}
\end{equation}
for some $\kappa>1$. Then if we call $c_{A,p_{0}}$ the smallest
of such $\kappa$, 
\[
\|Tf\|_{L^{A,1}}\leq2(C_{1}+C_{0}c_{A,p_{0}})\|f\|_{L^{A,1}}.
\]
\end{thm}
\begin{proof}
Let 
\[
f=f_{t}+f^{t}=f\chi_{\left\{ x\,:\,|f(x)|\leq f^{*}(t)\right\} }+f\chi_{\left\{ x\,:\,|f(x)|>f^{*}(t)\right\} }
\]
Then
\begin{align*}
\|Tf\|_{L^{p,1}} & =\int_{0}^{\infty}\frac{(Tf)^{*}(t)}{A^{-1}\left(\frac{1}{t}\right)}\frac{dt}{t}\leq2\int_{0}^{\infty}\frac{(Tf_{t})^{*}(t)}{A^{-1}\left(\frac{1}{t}\right)}\frac{dt}{t}+2\int_{0}^{\infty}\frac{(Tf^{t})^{*}(t)}{A^{-1}\left(\frac{1}{t}\right)}\frac{dt}{t}\\
 & =2(I+II)
\end{align*}
For $I$ note that
\[
(Tf_{t})^{*}(t)\leq(Tf_{t})^{*}(0)=\|Tf_{t}\|_{L^{\infty}}\leq C_{1}\|f_{t}\|_{L^{\infty}}\leq C_{1}f^{*}(t).
\]
Then
\[
\int_{0}^{\infty}\frac{(Tf_{t})^{*}(t)}{A^{-1}\left(\frac{1}{t}\right)}\frac{dt}{t}\leq C_{1}\int_{0}^{\infty}\frac{f^{*}(t)}{A^{-1}\left(\frac{1}{t}\right)}\frac{dt}{t}=C_{1}\|f\|_{L^{A,1}}
\]
For $II$ note that since
\[
|f^{t}(x)|\leq|f(x)|
\]
then $(f^{t})^{*}(s)\leq f^{*}(s)$. Furthermore $\mu_{f^{t}}(0)=\mu_{f}(f^{*}(t))\leq t$
which implies that $(f^{t})^{*}(t)=0$, and since $(f^{t})^{*}$ is
decreasing and non-negative this yields that $f^{*}(s)=0$ for $s\geq t$
. Hence
\begin{align*}
\int_{0}^{\infty}\frac{(Tf^{t})^{*}(t)}{A^{-1}\left(\frac{1}{t}\right)}\frac{dt}{t} & =\int_{0}^{\infty}\frac{1}{A^{-1}(\frac{1}{t})}t^{-\frac{1}{p_{0}}}t^{\frac{1}{p_{0}}}(Tf^{t})^{*}(t)\frac{dt}{t}\\
 & \leq\int_{0}^{\infty}\frac{1}{A^{-1}(\frac{1}{t})}t^{-\frac{1}{p_{0}}}\left(\sup_{s}s^{\frac{1}{p_{0}}}(Tf^{t})^{*}(s)\right)\frac{dt}{t}\\
 & \leq C_{0}\int_{0}^{\infty}\frac{1}{A^{-1}(\frac{1}{t})}t^{-\frac{1}{p_{0}}}\|f^{t}\|_{L^{p_{0},1}}\frac{dt}{t}\\
 & \leq C_{0}\int_{0}^{\infty}\frac{1}{A^{-1}(\frac{1}{t})}t^{-\frac{1}{p_{0}}}\int_{0}^{\infty}s^{\frac{1}{p_{0}}}(f^{t})^{*}(s)\frac{ds}{s}\frac{dt}{t}\\
 & =C_{0}\int_{0}^{\infty}\frac{1}{A^{-1}(\frac{1}{t})}t^{-\frac{1}{p_{0}}}\int_{0}^{t}s^{\frac{1}{p_{0}}}(f)^{*}(s)\frac{ds}{s}\frac{dt}{t}\\
 & =C_{0}\int_{0}^{\infty}s^{\frac{1}{p_{0}}}(f)^{*}(s)\left(\int_{s}^{\infty}\frac{1}{A^{-1}(\frac{1}{t})}t^{-\frac{1}{p_{0}}}\frac{dt}{t}\right)\frac{ds}{s}\\
 & \leq C_{0}c\int_{0}^{\infty}s^{\frac{1}{p_{0}}}\frac{s^{-\frac{1}{p_{0}}}}{A^{-1}(\frac{1}{s})}(f)^{*}(s)\frac{ds}{s}\\
 & =C_{0}c\int_{0}^{\infty}\frac{1}{A^{-1}(\frac{1}{s})}(f)^{*}(s)\frac{ds}{s}=\|f\|_{L^{A,1}}.
\end{align*}
where we used (\ref{eq:pwisecond}).
\end{proof}

\section{Proof of Theorem \ref{Thm:EndpointCoifman}\label{sec:ProofMain}}

Throughout this proof, since $\rho$ is fixed, we shall drop it when
denoting $B_{\rho}$ for the sake of clarity.

Given $u\in A_{1}$ and $v\in A_{\infty}$ we consider the operator
\begin{align*}
Sf(x) & =\frac{M(fu)(x)}{u(x)}\chi_{\{u\not=0\}}(x)
\end{align*}
Acually, we note that since $u\in A_{1}$ then $u>0$ a.e. and hence
the definition of $Sf$ makes sense a.e.

Observe that due to the fact that $u\in A_{1}$ then, $S$ is bounded
on $L^{\infty}(uv)$ with constant $C_{1}=[u]_{A_{1}}$. Now we show
that $S$ is bounded on $L^{p_{0}}(uv)$ for some $1<p_{0}<\infty$.
Note that
\[
\int_{\mathbb{R}^{n}}Sf(x)^{p_{0}}u(x)v(x)dx=\int_{\mathbb{R}^{n}}M(fu)(x)^{p_{0}}u(x)^{1-p_{0}}v(x)dx.
\]
Since $v\in A_{\infty}$ then we have that $v\in A_{t}$ for some
$t>1$. Then by the factorization theorem, there exist $v_{1},v_{2}\in A_{1}$
such that 
\[
v=v_{1}v_{2}^{1-t}.
\]
Hence 
\[
u^{1-p_{0}}v=v_{1}(uv_{2}^{\frac{t-1}{p_{0}-1}})^{1-p_{0}}.
\]
Now we observe that by Lemma there exists $0<\varepsilon_{0}<1$ depending
just on $[u]_{A_{1}}$ such that $uv_{2}^{\varepsilon}\in A_{1}$
for every $0<\varepsilon<\varepsilon_{0}$ and every $v_{2}\in A_{1}$.
Hence if we let 
\[
p_{0}=\frac{2(t-1)}{\varepsilon_{0}}+1
\]
we have that $u^{1-p_{0}}v\in A_{p_{0}}$ and $M$ is bounded on $L^{p_{0}}(u^{1-p_{0}}v)$
by Muckenhoupt theorem. Therefore, $S$ is bounded on $L^{p_{0}}$
with some constant $C_{0}$ that depends just on the $A_{1}$ constant
of $u$ and the $A_{t}$ constant of $v$.

Now we observe that for 
\[
B_{r}(t)=\frac{1}{\frac{1}{t^{r}}\log(e+\frac{1}{t^{r}})^{\rho}},
\]
then
\[
B_{r}^{-1}(t)=\frac{1}{\Phi_{\rho}^{-1}\left(\frac{1}{t}\right)^{\frac{1}{r}}}
\]
and we have that
\[
B_{r}^{-1}(t)\simeq t^{\frac{1}{r}}\log\left(e+\frac{1}{t}\right)^{\frac{\rho}{r}}
\]
with constant independent of $r$. Let us take $\Psi_{r'}$ such that
\[
t=B_{r}^{-1}(t)\Psi_{r'}^{-1}(t).
\]
Namely let $\Psi_{r'}^{-1}(t)=c_{\rho}\frac{t^{\frac{1}{r'}}}{\log\left(e+\frac{1}{t}\right)^{\frac{\rho}{r}}}.$
Note that for $r'>p_{0}$, if we can show that
\begin{equation}
\int_{s}^{\infty}t^{\frac{1}{r'}-\frac{1}{p_{0}}}\log(e+t)^{\frac{\rho}{r}}\frac{dt}{t}\leq c_{\rho}\frac{1}{\frac{1}{p_{0}}-\frac{1}{r'}}\frac{\log(e+s)^{\rho}}{\varepsilon s^{\varepsilon}}\label{eq:CheckHyPInt}
\end{equation}
namely that, following notation in Theorem \ref{Thm:Interpolation},
\[
C_{\Psi_{r'},p_{0}}\leq\frac{c_{\rho}}{\frac{1}{p_{0}}-\frac{1}{r'}},
\]
then, by Theorem \ref{Thm:Interpolation} itself, we will have that
\begin{equation}
\|Sf\|_{L^{\Psi_{r'},1}}\leq4\left(C_{1}+C_{0}\frac{c_{\rho}}{\frac{1}{p_{0}}-\frac{1}{r'}}\right)\|f\|_{L^{\Psi_{r'},1}}.\label{eq:Sf}
\end{equation}
Let us show then that $c_{\Psi_{r'},p_{0}}\leq\frac{c_{\rho}}{\frac{1}{p_{0}}-\frac{1}{r'}}$.
Let us call $\varepsilon=\frac{1}{r'}-\frac{1}{p_{0}}$. Then

\begin{align*}
\int_{s}^{\infty}t^{\frac{1}{r'}-\frac{1}{p_{0}}}\log(e+t)^{\frac{\rho}{r}}\frac{dt}{t} & =\int_{s}^{\infty}\log(e+t)^{\frac{1}{r}\rho}\frac{dt}{t^{1+\varepsilon}}=\int_{s}^{\infty}\frac{\log(e+t)^{\frac{1}{r}\rho}}{t^{\frac{1}{r}+\frac{\varepsilon}{r}}}\frac{dt}{t^{\frac{1}{r\text{'}}+\frac{\varepsilon}{r'}}}\\
 & =\int_{s}^{\infty}\left(\frac{\log(e+t)^{\rho}}{t^{\varepsilon+1}}\right)^{\frac{1}{r}}\frac{dt}{t^{\frac{1}{r'}+\frac{\varepsilon}{r\text{'}}}}\\
 & \leq\left(\int_{s}^{\infty}\frac{\log(e+t)^{\rho}}{t^{\varepsilon}}\frac{1}{t}dt\right)^{\frac{1}{r}}\left(\int_{s}^{\infty}\frac{dt}{t^{1+\varepsilon}}\right)^{\frac{1}{r'}}.
\end{align*}
Note that for the second integral 
\[
\int_{s}^{\infty}\frac{dt}{t^{1+\varepsilon}}=\frac{1}{\varepsilon s^{\varepsilon}}
\]
and for the first, 
\begin{align*}
\int_{s}^{\infty}\frac{\log(e+t)^{\rho}}{t^{\varepsilon}}\frac{1}{t}dt & =\left[\frac{t^{-\varepsilon}}{-\varepsilon}\log(e+t)^{\rho}\right]_{t=s}^{\infty}+\int_{s}^{\infty}\rho\frac{\log(e+t)^{\rho-1}}{(e+t)t^{\varepsilon}}dt\\
 & \leq\frac{\log(e+s)^{\rho}}{\varepsilon s^{\varepsilon}}+\rho\int_{s}^{\infty}\frac{\log(e+t)^{\rho-1}}{t^{1+\varepsilon}}dt.
\end{align*}
If $\rho-1\leq0$ then
\[
\int_{s}^{\infty}\rho\frac{\log(e+t)^{\rho-1}}{t^{1+\varepsilon}}dt\leq\rho\int_{s}^{\infty}\frac{1}{(e+t)t^{\varepsilon}}dt\leq\rho\int_{s}^{\infty}\frac{1}{t^{1+\varepsilon}}dt=\frac{\rho}{\varepsilon s}
\]
and we are done. Otherwise arguing again in the same way for the second
term until $\rho-k\leq0$ we have that 
\[
\int_{s}^{\infty}t^{\frac{1}{r'}-\frac{1}{p_{0}}}\log(e+t)^{\frac{\rho}{r}}\frac{dt}{t}\leq c_{\rho}\frac{\log(e+s)^{\rho}}{\varepsilon s^{\varepsilon}}.
\]
Combining the estimates above yields \ref{eq:CheckHyPInt} as we wanted
to show. Having (\ref{eq:Sf}) at our disposal we continue our argument
noting that

\[
\|Sf\|_{L^{\Psi_{r'},1}(uv)}\leq K_{0}\|f\|_{L^{\Psi_{r'},1}(uv)}
\]
for every $r'\geq2p_{0}$ with $K_{0}=8p_{0}(C_{0}+C_{1})$.

Observe that for every $W_{1}\in A_{1}$ with $[W_{1}]_{A_{1}}\leq2K_{0}$
there exists $0<\tilde{\varepsilon_{0}}<1$ depending just on $K_{0}$
such that $W_{1}W_{2}^{\varepsilon}\in A_{1}$ for every $0<\varepsilon<\tilde{\varepsilon_{0}}$.
Let us fix $0<\varepsilon<\min\left\{ \tilde{\varepsilon}_{0},\frac{1}{2p_{0}}\right\} $
and let $r=\left(\frac{1}{\varepsilon}\right)'$. Then $r'>2p_{0}$
and hence $S$ is bounded on $L^{\Psi_{r'},1}(uv)$. Now we use the
Rubio de Francia algorithm to define the operator 
\[
Rf(x)=\sum_{k=0}^{\infty}\frac{S^{k}h(x)}{2^{k}K_{0}^{k}}.
\]
It follows from this definition that
\begin{enumerate}
\item $h(x)\leq Rh(x)$
\item $\|Rh\|_{L^{\Psi_{r'},1}(uv)}\leq2\|h\|_{L^{\Psi_{r'},1}(uv)}$
\item $S(Rh)(x)\leq2K_{0}Rh(x)$
\end{enumerate}
In particular it follows from the definition of $S$ that 
\[
Rhu\in A_{1}
\]
with $[Rhu]_{A_{1}}\le2K_{0}$. Now let $W_{1}=Rhu$ and $W_{2}=v_{1}\in A_{1}$
(bear in mind that $v=v_{1}v_{2}^{1-t})$. Then $W_{1}W_{2}^{\varepsilon}\in A_{1}$
and we have that
\[
Rhuv^{1/r'}=W_{1}W_{2}^{\varepsilon}v_{2}^{(1-t)/r'}\in A_{\infty}.
\]
Let $(f,g)\in\mathcal{F}$ such that the left hand side of (\ref{eq:Hip})
is finite. Then taking into account Lemmas \ref{Lem:rescaling} and
\ref{lem:assoc} we have that
\begin{align*}
\left\Vert \frac{f}{v}\right\Vert _{L^{B,\infty}(uv)}^{\frac{1}{r}} & =\left\Vert \left(\frac{|f|}{v}\right)^{\frac{1}{r}}\right\Vert _{L^{B_{r},\infty}(uv)}\\
 & \simeq\sup_{\|h\|_{L^{\Psi_{r'},1}(uv)}=1}\left|\int_{\mathbb{R}^{n}}\left(\frac{|f|}{v}\right)^{\frac{1}{r}}h(x)u(x)v(x)dx\right|
\end{align*}
We fix $h$ with $\|h\|_{L^{\Psi_{r'},1}(uv)}=1$ and we shall assume
that it is non negative. We are going to use (\ref{eq:Hip}) with
the weight $Rhuv^{\frac{1}{r'}}\in A_{\infty}$ and $p=\frac{1}{r}$.
Note that we may choose $p$ arbitrarily since by extrapolation results
contained in \cite{CUMPAinfty}, if (\ref{eq:Hip}), then it holds
as well for every $0<p<\infty$. To use (\ref{eq:Hip}) we also need
to have that the left hand side is finite. Now we show that it is
actually the case.
\begin{align*}
\left|\int_{\mathbb{R}^{n}}\left(\frac{|f|}{v}\right)^{\frac{1}{r}}h(x)u(x)v(x)dx\right| & \lesssim\left\Vert \left(\frac{|f|}{v}\right)^{\frac{1}{r}}\right\Vert _{L^{B_{r},\infty}(uv)}\left\Vert Rh\right\Vert _{L^{\Psi_{r'},1}(uv)}\\
 & \lesssim\left\Vert \left(\frac{|f|}{v}\right)^{\frac{1}{r}}\right\Vert _{L^{B_{r},\infty}(uv)}\left\Vert h\right\Vert _{L^{\Psi_{r'},1}(uv)}<\infty
\end{align*}
Then we have that 
\begin{align*}
\left|\int_{\mathbb{R}^{n}}\left(\frac{|f|}{v}\right)^{\frac{1}{r}}h(x)u(x)v(x)dx\right| & \leq\int|f(x)|^{\frac{1}{r}}Rh(x)u(x)v(x)^{\frac{1}{r'}}dx\\
 & \lesssim\int|g(x)|^{\frac{1}{r}}Rh(x)u(x)v(x)^{\frac{1}{r'}}dx\\
 & \lesssim\left\Vert \left(\frac{|g|}{v}\right)^{\frac{1}{r}}\right\Vert _{L^{B_{r},\infty}(uv)}\left\Vert Rh\right\Vert _{L^{\Psi_{r'},1}(uv)}\\
 & \lesssim\left\Vert \frac{g}{v}\right\Vert _{L^{B,\infty}(uv)}^{\frac{1}{r}}\left\Vert h\right\Vert _{L^{\Psi_{r'},1}(uv)}
\end{align*}
and hence we are done.

\end{document}